\documentclass[12pt]{amsart}

\usepackage{amssymb}
\usepackage{amsmath}
\usepackage{amsfonts}
\usepackage{graphicx}
\usepackage{amsthm}
\usepackage{enumerate}
\usepackage[mathscr]{eucal}
\usepackage{mathrsfs}
\usepackage{verbatim}
\usepackage{booktabs}
\usepackage{enumitem}

\makeatletter
\@namedef{subjclassname@2020}{%
  \textup{2020} Mathematics Subject Classification}
\makeatother

\numberwithin{equation}{section}

\theoremstyle{plain}
\newtheorem{theorem}{Theorem}[section]
\newtheorem{lemma}[theorem]{Lemma}
\newtheorem{proposition}[theorem]{Proposition}

\theoremstyle{plain}

\theoremstyle{remark}
\newtheorem{remark}[theorem]{Remark}

\DeclareMathOperator{\supp}{supp}
\DeclareMathOperator{\dist}{dist}

\DeclareMathOperator{\diag}{diag}

\allowdisplaybreaks[4]

\begin{document}

\date{}
\title
[Lattice Points in stretched finite type domains]{Lattice points in stretched finite type domains}
\author[J. Guo and T. Jiang]{Jingwei Guo \  \ Tao Jiang}

\address{Department of Mathematics\\
University of Science and Technology of China\\
Hefei 230026, China}
\email{jwguo@ustc.edu.cn}

\address{Department of Mathematics\\
University of Science and Technology of China\\
Hefei 230026, China}
\email{jt1023@mail.ustc.edu.cn}

\thanks{J.G. is partially supported by the Fundamental Research Funds for the Central Universities (No. WK3470000013) and NSF of Anhui Province, China (No. 2108085MA12).}



\date{}

\begin{abstract}
We study an optimal stretching problem, which is a variant lattice point problem, for convex domains in $\mathbb{R}^d$ ($d\geq 2$) with smooth boundary of finite type that are symmetric with respect to each coordinate hyperplane/axis. We prove that optimal domains which contain the most positive (or least nonnegative) lattice points are asymptotically balanced.
\end{abstract}

\maketitle

\section{Introduction}\label{introduction}

The classical lattice point problem is about counting the number of lattice points $\mathbb{Z}^d$ in large domains
in the Euclidean space $\mathbb{R}^d$. It has a long history which can be traced back to C.F. Gauss who studied the number of lattice points in large disks.  In this paper we study the following variant lattice point problem, the so-called \textit{optimal stretching problem}.

Let
\begin{equation*}
A=\diag (a_1,a_2\ldots,a_d)
\end{equation*}
be a positive definite diagonal matrix with determinant  $1$. Let $\Omega\subset \mathbb{R}^d$ be a compact domain which contains the origin in its interior. A volume-preserving \textit{stretch} of $\Omega$ by the \textit{stretching factor} $A$ is a domain of the form
\begin{equation*}
A\Omega=\{(a_1x_1,\ldots,a_dx_d):(x_1,\ldots,x_d)\in\Omega\}.
\end{equation*}
One would like to know the limiting behaviour of $A$ (as $t$ goes to infinity) for those matrices $A$ such that the number of positive-integer lattice points in the enlarged stretch of $\Omega$, i.e. $\#(\mathbb{N}^d\cap tA\Omega)$, attains the largest value. A similar question can be asked for matrices $A$ such that $\#(\mathbb{Z}_+^d\cap t A\Omega$) attains the smallest value where $\mathbb{Z}_{+}=\{0\}\cup \mathbb{N}$.

The optimal stretching problem was initiated by Antunes and Freitas, who considered in \cite{AF13} the stretch of the unit disk in $\mathbb{R}^2$ and proved that among all ellipses of the same area, those that enclose the most lattice points in the first quadrant must be more and more ``round'', as the area goes to infinity. In other words, the limit of the stretching factor is the identity matrix. Their motivation of such a study was a problem in spectral theory of minimizing Dirichlet eigenvalues of the Laplace operator among rectangles of equal area. In fact their result on asymptotically minimizing the $n$-th eigenvalue among rectangles of given area is equivalent to asymptotically maximizing the number of positive-integer lattice points in ellipses of given area.

We remark that the optimal stretching problem and closely related shape/eigenvalue optimization problems in spectral theory have been of large interest in recent years. For explanation on their connection and more results on the latter problems see for example van den Berg, Bucur and Gittins~\cite{BBG}, van den Berg and Gittins~\cite{BG}, Gittins and Larson~\cite{GL}, Larson~\cite{Larson19, LarsonJST} and references therein. In what follows we focus on the optimal stretching problem for domains more general than ellipses/ellipsoids.

In a pair of papers Laugesen and Liu~\cite{LL16} and Ariturk and Laugesen~\cite{AL} extended the result of Antunes and Freitas by considering general planar domains (including $p$-ellipses $|sx|^p+|y/s|^p=t^p$ for $p\in (0,\infty)\setminus\{1\}$). They showed, among others, that under mild assumptions on the boundary curve  optimal domains which contain the most positive (or least nonnegative) lattice points must be asymptotically balanced.  (We recall that a domain in $\mathbb{R}^d$ ($d\geq 2$) is said to be \emph{balanced} if the $(d-1)$-dimensional measures of the intersections of the domain with each coordinate hyperplane are equal.) They also provided rates of convergence of optimal stretching factors. Notice that their results allow the curvature of the boundary curve to vanish or blow up at the intersection points with coordinate axes.

However, if the boundary is ``too flat'' the result could be very different---optimal domains needs not to be asymptotically balanced. For example, Marshall and Steinerberger~\cite{MS18} analyzed the case of triangles (namely the $p$-ellipses with $p=1$). They showed that there are infinitely many optimal domains for arbitrarily large $t$.

The difference between these results is essentially a consequence of different curvature assumptions. This is not surprising since the lattice point counting is closely related to oscillatory integral estimates in which curvature plays a key role. The phenomenon of asymptotic balancing was further confirmed in Marshall~\cite{M20} for convex domains in $\mathbb{R}^d$ with $C^{d+2}$ boundary and non-vanishing Gaussian curvature.

Naturally one may next ask if asymptotic balancing still occurs for the intermediate case between ``non-flat'' and ``flat'' cases, especially in high dimensions. A few attempts have been made. For example, in \cite{GW18} the first author and Wang considered certain \textit{special} convex domains of finite type\footnote{That is, at each boundary point each tangent line has finite order of contact.} in $\mathbb{R}^d$ (including super spheres, i.e. high dimensional $p$-ellipsoids) and gave an affirmative answer.   Later we slightly generalized this result in \cite{GJ19}.

The goal of this paper is to prove the aforementioned asymptotic balancing phenomenon for \textit{arbitrary} convex domains of finite type.

Let $\Omega$ be a convex domain with smooth boundary of finite type. Throughout this paper we set, for any  $P\in\partial\Omega$ , that
\begin{equation}\label{def-nuomegaP}
  \nu_{\Omega}(P)=\sum_{i=1}^{d-1}\mathfrak a_i^{-1}
\end{equation}
and
\begin{equation}\label{def-nuomegaP2}
\begin{split}
  \nu_{\Omega}^{(2)}(P)=\left\{
  \begin{array}{ll}
  0&\text{if $d=2$},\\
  \sum_{i=2}^{d-1}\mathfrak a_i^{-1}&\text{if $d\ge 3$},
  \end{array}
  \right.
\end{split}
\end{equation}
where $\mathfrak a=(\mathfrak a_1,\mathfrak a_2,\ldots,\mathfrak a_{d-1})$ is the multitype (type if $d=2$) of $\partial \Omega$ at the point $P$.  See Iosevich, Sawyer and Seeger~\cite[P. 155--156]{ISS02} for the definition of multitype. We also set
\begin{equation}
\nu_{\Omega}=\min_{P\in\partial \Omega}\nu_{\Omega}(P) \label{222}
\end{equation}
and
\begin{equation}
\mu_{\Omega}=\frac{1}{2}+\min_{P\in\partial \Omega }\nu_{\Omega}^{(2)}(P). \label{333}
\end{equation}

For each $t\geq 1$ we define
\begin{equation}\label{def-At}
\mathfrak A_{\Omega}(t)=\mathrm{argmax}_A\#\left(\mathbb{N}^d\cap tA\Omega\right)
\end{equation}
and
\begin{equation}\label{def-At222}
\widetilde{\mathfrak {A}}_{\Omega}(t)=\mathrm{argmin}_{A}\#\left(\mathbb{Z}_+^d\cap tA\Omega\right),
\end{equation}
where the argmax and argmin  range over all positive definite diagonal matrices $A$ of determinant $1$. The notation  $\mathrm{argmax}_x f(x)$ (resp. $\mathrm{argmin}_x f(x)$) is the set of points $x$
for which $f(x)$ attains the function's largest (resp. least) value. Note that optimal stretching factors in \eqref{def-At} and \eqref{def-At222} are in general not unique. In what follows, when we write $A(t)$ in $\mathfrak A_{\Omega}(t)$ (resp. $\widetilde{\mathfrak {A}}_{\Omega}(t)$), we really mean that $A(t)$ is an arbitrary element in $\mathfrak A_{\Omega}(t)$ (resp. $\widetilde{\mathfrak {A}}_{\Omega}(t)$).

For each $1 \le j \le d$, we let $\Omega_j$ be the intersection of $\Omega$ with the coordinate hyperplane $x_j=0$.

With the above notations, our main results can be stated as follows.

\begin{theorem}\label{main1}
Let $\Omega\subset \mathbb R^{d}$ ($d\geq 2$) be a convex compact domain, that is symmetric with respect to each coordinate hyperplane (axis if $d=2$), with smooth boundary of finite type. If
\begin{equation*}
A(t)=\diag (a_1(t),\ldots,a_d(t))\in\mathfrak A_{\Omega}(t),
\end{equation*}
then
\begin{equation}
    \left|a_j(t)-\frac{|\Omega_j|}{\sqrt[d]{|\Omega_1||\Omega_2|\cdots|\Omega_d|}}\right|=O\left(t^{-\gamma}\right),\quad 1\le j\le d,   \label{optpositive}
\end{equation}
where $|\Omega_j|$ is the $(d-1)$-dimensional measure of $\Omega_j$ and
\begin{equation*}
    \gamma=\min\left\{\frac{\nu_{\Omega}}{2},\frac{\mu_{\Omega}}{2(d-\mu_{\Omega})}\right\}.
\end{equation*}
Similarly, if
\begin{equation*}
\widetilde A(t)=\diag (\tilde a_1(t),\ldots,\tilde a_d(t))\in \widetilde{\mathfrak A}_{\Omega}(t),
\end{equation*}
then
\begin{equation}
\left|\tilde a_j(t)-\frac{|\Omega_j|}{\sqrt[d]{|\Omega_1||\Omega_2|\cdots|\Omega_d|}}\right|
=O\left(t^{-\gamma}\right),\quad 1\le j\le d.   \label{optnonnegative}
\end{equation}
\end{theorem}

\begin{remark}
The key to prove this theorem is an application of a delicate estimate of the Fourier transform of surface carried measure obtained in Iosevich, Sawyer and Seeger~\cite{ISS02}.

Our main goal was to weaken curvature assumptions in high dimensions, namely to extend the results in \cite{M20, GW18, GJ19} to arbitrary finite type domains. A further interesting question is whether the optimal balancing still occurs for infinite type domains. This may be a hard question noticing that there are very few results on counting lattice points in general convex domains of infinite type.

For convenience of stating our results we assume the domain's boundary is smooth. However it suffices to assume sufficient smoothness.

Our results work for finite type domains in $\mathbb{R}^2$. We did not try to further weaken assumptions however. Comparing to the planar results in \cite{LL16, AL}, we allow the curvature to vanish at finitely many boundary points rather than just at points of intersection with coordinate axes. The results in \cite{LL16, AL} have weaker regularity assumptions and are good for both convex and concave cases.
\end{remark}

{\it Notations:} The Fourier transform of any function $f\in L^1(\mathbb{R}^d)$ is
$\widehat{f}(\xi)=\int f(x) \exp(-2\pi i x\cdot \xi) \, dx$. For
functions $f$ and $g$ with $g$ taking nonnegative real values,
$f\lesssim g$ means $|f|\leq Cg$ for some constant $C$. If $f$ is
nonnegative, $f\gtrsim g$ means $g\lesssim f$. The Landau notation
$f=O(g)$ is equivalent to $f\lesssim g$. The notation $f\asymp g$
means that $f\lesssim g$ and $g\lesssim f$.  We set $\mathbb{R}^d_*=\mathbb{R}^d\setminus \{0\}$ and $\mathbb{Z}^d_*=\mathbb{Z}^d\setminus \{0\}$.


\section{Lattice point counting}

Throughout this section, we denote by $A=\diag(a_1,\ldots,a_d)$  a positive definite diagonal matrix with determinant $1$ and
\begin{equation*}
a_* =\|A^{-1}\|_\infty =\max\{a_1^{-1}, \ldots, a_d^{-1}\}.
\end{equation*}

We first quote a result from \cite{GW18} on two-term bounds for lattice point counting, which generalizes \cite[Proposition 6 and 9]{LL16} to the setting of strictly convex\footnote{A domain $\Omega\subset \mathbb{R}^d$ is said to be strictly convex if the line segment connecting any two points $x$ and $y$ in $\Omega$ lies in the interior of $\Omega$, except possibly for its endpoints.} domains in $\mathbb{R}^d$. We can indeed apply this result in this paper since convex domains of finite type are strictly convex.

\begin{lemma}[{\cite[Proposition 2.1]{GW18}}]\label{general}
Let $\Omega\subset [-C,C]^d\subset \mathbb{R}^d$ ($d\geq 2$, $C>0$) be strictly convex, compact and symmetric with respect to each coordinate hyperplane (axis if $d=2$) with $C^2$ boundary. There is a positive constant $c$ depending only on the domain $\Omega$ such that if $t/a_*\ge 1/C$ then
\begin{equation*}
    \#\left(\mathbb N^d\cap tA\Omega\right)\le 2^{-d}|\Omega|t^d-ca_*t^{d-1}
\end{equation*}
and
\begin{equation*}
 \#\left(\mathbb Z_+^d\cap tA\Omega\right)\ge 2^{-d}|\Omega|t^d+ca_*t^{d-1}.
\end{equation*}
\end{lemma}

We next quote some known results on the decay of the Fourier transform of surface carried measure. In order to state them we briefly recall some notations from \cite[P. 155--156]{ISS02} that are related to the definition of multitype (see also \cite[P. 1270]{schulz}). For a convex compact domain $\Omega\subset\mathbb R^d$ $(d\geq 2)$ with smooth boundary of finite type and an arbitrarily fixed $P\in \partial \Omega$, denote by  $T_P(\partial\Omega)$ the tangent plane (line if $d=2$) of $\partial\Omega$ at $P$. Let $S_P^{m_j}$, $1\leq j\leq k$, be the flag of subspaces of $T_P(\partial\Omega)$, and $W_j$ the orthogonal complement of $S_P^{m_j}$ in $S_P^{m_{j-1}}$, as defined in \cite[P. 155--156]{ISS02}. One can choose an orthonormal basis $\{V_1,\ldots,V_{d-1}\}$ of $T_P(\partial\Omega)$ such that for any $V=\sum_{i=1}^{d-1}x_iV_i $, the equality
\begin{equation*}
\left|\Pi_j^P V\right|^2=\sum_{i=d-\dim S_P^{m_{j-1}}}^{d-1-\dim S_P^{m_j}}x_i^2
\end{equation*}
holds, where $\Pi_j^P$ represents the orthogonal projection on $T_P(\partial \Omega)$ to $W_j$. Here $|\cdot|$ denotes the Euclidean distance in  $W_j$. We notice that  $\left|\Pi_j^P V\right|$ is independent of the choice of the orthonormal basis, hence we can apply \cite[Proposition 1.2]{ISS02} with the above particularly chosen basis. Let $n(P)$ denote the unit exterior normal of $\partial \Omega$ at $P$. We may assume the basis $\{V_1, \ldots,V_{d-1},-n(P)\}$ has the same orientation as $\{e_1,\ldots,e_d \}$. There exists a rotation matrix $\mathrm O=\mathrm{O}(P)$ such that
\begin{equation}\label{def-mathrmO}
    (e_1,\ldots,e_d)=(V_1, \ldots,V_{d-1},-n(P))\mathrm O,
\end{equation}
namely
\begin{equation*}
e_j=\sum_{i=1}^{d-1} \textrm{o}_{ij}V_i-\textrm{o}_{dj} n(P),\quad \textrm{where $\mathrm O=(o_{ij})$}.
\end{equation*}

Let $\mathrm d\sigma$ be the surface measure carried on $\partial \Omega$. The following decay of its Fourier transform is known.

\begin{lemma}[\cite{BNW, randol-1969, ISS02}] \label{decay}
Let $\Omega\subset \mathbb R^d$ $(d\geq 2)$ be a convex compact domain with smooth boundary of finite type and $P\in\partial \Omega$.  Then there is a neighborhood $U_P\subset\partial\Omega$ of $P$ and a conic neighborhood $ V_P\subset\mathbb R^d_*$ of $\{\pm n(P)\}$ such that for all $\chi\in C_0^{\infty}(U_P)$ and all $\xi\in V_P$, we have
\begin{equation*}
\left|\widehat{\chi\mathrm d\sigma}(\xi)\right|\lesssim \min\left\{|\xi|^{-\nu_{\Omega}(P)},|\xi|^{-\frac{1}{2}
-\nu_{\Omega}^{(2)}(P)}\left(\sum_{i=1}^{d-1}
\left(\frac{|\mathrm O_i\xi|}{|\xi|}\right)^{\frac{\mathfrak a_i}{\mathfrak a_i-1}}\right)^{\frac{1}{\mathfrak a_1}-\frac{1}{2}}\right\},
\end{equation*}
where $\nu_{\Omega}(P)$ and $\nu_{\Omega}^{(2)}(P)$ are defined by \eqref{def-nuomegaP} and \eqref{def-nuomegaP2} respectively, $\mathrm O_i$ is the $i$-th row vector of the matrix $\mathrm O=\mathrm{O}(P)$ defined by \eqref{def-mathrmO}, $\mathfrak a=(\mathfrak a_1,\ldots,\mathfrak a_{d-1})$ is the multitype (type if $d=2$) of $\partial \Omega$
at $P$, and the implicit constant may depend on the domain $\Omega$ and  upper bounds of $\chi$ and finitely many derivatives of $\chi$.
\end{lemma}

The first bound on the right hand side is standard, which follows easily from  \cite[P. 335--336, Theorem B]{BNW}. The second one follows from \cite[Lemma 1]{randol-1969} in dimension two and \cite[Proposition 1.2]{ISS02} in higher dimensions.

In the rest of this section we establish results on lattice point counting in stretched finite type domains. Recall that $\nu_{\Omega}$ and $\mu_{\Omega}$ are defined by \eqref{222} and \eqref{333} respectively.

\begin{proposition}\label{LPP}
 Let $\Omega\subset [-C,C]^d\subset\mathbb R^d$ ($d\geq 2$, $C>0$) be a convex compact domain, which contains the origin as an inner point,  with smooth boundary of finite type. If $t/a_*\geq 1/C$,  then
\begin{equation}\label{eLPP}
\#\left(\mathbb Z^d\cap tA\Omega\right)=|\Omega|t^d+O\left(a_*^{d^2-d+1}\left(t^{d-1-\nu_{\Omega}}
+t^{d-1-\frac{\mu_{\Omega}}{d-\mu_{\Omega}}}\right)\right),
\end{equation}
where the implicit constant depends only on the domain $\Omega$.
\end{proposition}

\begin{remark}
If $A$ is a fixed matrix, the above result is given directly by \cite[Theorem 1.3]{ISS02}. For our need, $A$ is allowed to change. Hence we have to track the impact of $A$ and modify the proof of \cite[Theorem 1.3]{ISS02} accordingly.

We did not try to find the smallest exponent of the $a_*$ term since it does not matter in the study of the optimal stretching problem. Indeed, we will manage to show that $a_*$ is uniformly bounded in Section \ref{sec3} hence the $a_*$ term is bounded by a constant after all.
\end{remark}

\begin{proof}[Proof of Proposition \ref{LPP}]
Let $0\leq \rho\in C_0^{\infty}(\mathbb{R}^d)$ be a cut-off function with
$\supp \rho\subset B(0, 1)$ and $\int_{\mathbb{R}^d}\rho(x) \, \textrm{d}x=1$. Set $\rho_{\varepsilon}(x)= \varepsilon^{-d}\rho(\varepsilon^{-1}x)$, $0<\varepsilon<1$, and
\begin{equation*}
N_{A,\varepsilon}(t)=\sum_{k\in\mathbb{Z}^d}\chi_{tA
\Omega}\ast\rho_{\varepsilon}(k),
\end{equation*}
where $\chi_{tA\Omega}$ denotes the characteristic function of $tA\Omega$. It is a standard result that there exists a constant  $c>0$ depending only on the domain $\Omega$ such that
\begin{equation}
 N_{A,\varepsilon}(t-ca_*\varepsilon) \leq\#\left(\mathbb Z^d\cap tA\Omega\right)
 \leq N_{A,\varepsilon}(t+ca_*\varepsilon). \label{777}
 \end{equation}

By using the Poisson summation formula we have
 \begin{equation}
     N_{A,\varepsilon}(t)=t^{d}\sum_{k\in\mathbb Z^d}\widehat{\chi_{\Omega}}(tAk)\widehat\rho(\varepsilon k)=|\Omega|t^d+R_{A,\varepsilon}(t) \label{888}
 \end{equation}
with
\begin{equation*}
R_{A,\varepsilon}(t)=t^d\sum_{k\in\mathbb{Z}_\ast^d}\widehat{\chi_{\Omega}}(tAk)\widehat{\rho}(\varepsilon k).
\end{equation*}

Let $\Gamma$ denote the set of points $P\in\partial\Omega$ at which all principal curvatures vanish.
It is known that $\Gamma$ is a finite set (see \cite[P. 164]{ISS02}). For each $P\in  \Gamma$, choose an open conic symmetric neighborhood $V_P$  of the normals $\{\pm n(P)\}$. If two points in $\Gamma$ have parallel normals we choose the same conic neighborhood for both of them. We may shrink these neighborhoods so that they are disjoint pairwise.

Let $\dist_{\infty}$ denote the distance taken with respect to the $\ell^{\infty}$ metric in $\mathbb R^d$. For $P\in\Gamma$ let
 \begin{equation*}
     \mathfrak N^1_P=\{x\in V_P:\dist_{\infty}(x,\mathbb{R}n(P))\leq 3/4\},
 \end{equation*}
  \begin{equation*}
 \mathfrak N^2_P=\{x\in V_P: \dist_{\infty}(x,\mathbb Rn(P))> 3/4\}
 \end{equation*}
 and
 \begin{equation*}
 \mathfrak{M}=\{x\in \mathbb{R}^d_* : x\notin \cup_{p\in\Gamma} V_P \}.
 \end{equation*}

To estimate $R_{A,\varepsilon}(t)$ we just need to estimate
\begin{equation*}
S_P^1=t^d\sum_{k\in A^{-1}\mathfrak N^1_P}\widehat{\chi_{\Omega}}(tAk)\widehat\rho(\varepsilon k),
\end{equation*}
\begin{equation*}
S_P^2=t^d\sum_{k\in A^{-1}\mathfrak N^2_P}\widehat{\chi_{\Omega}}(tAk)\widehat\rho(\varepsilon k)
\end{equation*}
and
\begin{equation*}
S_0=t^d\sum_{k\in A^{-1}\mathfrak{M}}\widehat{\chi_{\Omega}}(tAk)\widehat\rho(\varepsilon k).
\end{equation*}

To the sum $S_P^1$ we apply the bound
\begin{equation}
\widehat{\chi_{\Omega}}(\xi)\lesssim |\xi|^{-1-\nu_{\Omega}} \quad \textrm{if $\xi\in\mathfrak N^1_P$}.\label{111}
\end{equation}
To verify this bound, by the divergence theorem, we have
\begin{equation*}
\widehat{\chi_{\Omega}}(\xi)= \frac{i}{|\xi|} \int_{\partial\Omega} \frac{\xi}{|\xi|}\cdot n(x)e^{-2\pi i \langle x, \xi \rangle }\,\textrm{d}\sigma(x).
\end{equation*}
Let $P'\in \partial\Omega$ be the boundary point whose outward normal is along $-n(P)$. Following a standard argument from the oscillatory integral theory, we split the above integral over $\partial\Omega$ into three parts over a neighborhood about $P$, a neighborhood about $P'$ and the rest respectively. The former two parts are of size $O(|\xi|^{-\nu_{\Omega}})$, both yielded by the first bound of Lemma \ref{decay} (we may shrink the conic neighborhood $V_p$ if necessary). The third part is of size $O(|\xi|^{-N})$, given by a simple integration by parts. The bound \eqref{111} then follows easily.

Applying \eqref{111} yields
\begin{equation*}
S_P^1 \lesssim t^{d-1-\nu_{\Omega}}\sum_{k\in A^{-1}\mathfrak N^1_{P}}|Ak|^{-1-\nu_{\Omega}}.
\end{equation*}
We split the above sum on the right into two sums depending on whether $|\langle n(P), Ak\rangle |$ is $>C$ or $\leq C$ for an absolute constant $C$. If $C$ is large, a comparison with an integral yields that the sum with $|\langle n(P), Ak\rangle |>C$ is
\begin{equation*}
\lesssim a_*^d \int_{\mathcal{T}}|x|^{-1-\nu_{\Omega}}\,\mathrm{d}x \lesssim a_*^d,
\end{equation*}
where $\mathcal{T}$ represents a tubular neighborhood of a line away from the origin. Trivial estimate gives that the sum with $|\langle n(P), Ak\rangle |\leq C$  is
\begin{equation*}
\lesssim a_*^{d-1}a_*^{1+\nu_{\Omega}}=a_*^{d+\nu_{\Omega}}.
\end{equation*}
Therefore
\begin{equation}
S_P^1\lesssim a_*^{d+\nu_{\Omega}}t^{d-1-\nu_{\Omega}}. \label{444}
\end{equation}


To the sum $S_0$ we apply the bound
\begin{equation*}
\widehat{\chi_{\Omega}}(\xi)\lesssim |\xi|^{-1-\mu_{\Omega}} \quad \textrm{if $\xi\in\mathfrak{M}$},
\end{equation*}
which follows from the divergence theorem, the Bruna-Nagel-Wainger estimate (in \cite{BNW}) and an integration by parts argument. Hence
\begin{align}
S_0 &\lesssim t^d\sum_{k\in A^{-1}\mathfrak{M}}|tAk|^{-1-\mu_{\Omega}}\left|\widehat\rho(\varepsilon k)\right|\nonumber\\
     &\lesssim a_*^{1+\mu_{\Omega}}t^{d-1-\mu_{\Omega}}\varepsilon^{1+\mu_{\Omega}-d}. \label{666}
 \end{align}


For the sum $S_P^2$ we handle $\widehat{\chi_{\Omega}}(\xi)$, $\xi\in\mathfrak N^2_P$, similarly as in the proof of \eqref{111},
except that we now use the second bound of Lemma \ref{decay}. As before, the estimate of $\widehat{\chi_{\Omega}}$ is reduced to the Fourier transform of the surface carried measure $\textrm{d}\sigma$, which is then split into three parts. We apply the second bound of Lemma \ref{decay} to the first part (over a neighborhood about $P$).
We may assume $P'\notin \Gamma$ without loss of generality, thus the second part (over a neighborhood about $P'$) is of size $O(|\xi|^{-\mu_{\Omega}})$ by the Bruna-Nagel-Wainger estimate. The third part is of size $O(|\xi|^{-N})$ by integration by parts. To conclude we obtain the bound
\begin{equation*}
\widehat{\chi_{\Omega}}(\xi)\lesssim |\xi|^{-1
    -\mu_{\Omega}}\left(\sum_{l=1}^{d-1}\left(\frac{|\mathrm O_l\xi|}{|\xi|}\right)^{\mathfrak a_l'}\right)^{\frac{1}{\mathfrak a_1}-\frac{1}{2}} \quad \textrm{if $\xi\in\mathfrak N^2_P$},
\end{equation*}
where $\mathfrak a_l'=\mathfrak a_l/(\mathfrak a_l-1)$ with $(\mathfrak a_1,  \ldots, \mathfrak a_{d-1})$ the multitype (type if $d=2$) of $\partial \Omega$ at $P$, and $\mathrm O_l$ is the $l$-th row vector of the matrix $\mathrm O=\mathrm{O}(P)$ defined by \eqref{def-mathrmO}.

Applying the above bound yields
\begin{equation*}
S_P^2 \lesssim   t^{d-1
    -\mu_{\Omega}}\!\!\sum_{k\in A^{-1}\mathfrak N^2_P}\!\!|Ak|^{-1
    -\mu_{\Omega}}\left(\sum_{l=1}^{d-1}\left(\frac{|\mathrm O_lAk|}{|Ak|}\right)^{\mathfrak a_l'}\right)^{\!\!\!\frac{1}{\mathfrak a_1}-\frac{1}{2}}\!\!(1+|\varepsilon k|)^{-N}.
\end{equation*}
Notice that
\begin{equation*}
4\le\mathfrak a_1\le \mathfrak a_2\le\cdots\le \mathfrak a_{d-1},
\end{equation*}
\begin{equation*}
|k|=|A^{-1}Ak|\ge a_*^{1-d}|Ak|
\end{equation*}
 and if $\xi\in \mathfrak N^2_P$ (i.e. $\xi\in V_P$ and $\dist_{\infty}(\xi,\mathbb Rn(P))> 3/4$) and $V_P$ is sufficiently narrow then $|\xi|\geq 1$. A dyadic decomposition on the size of $|Ak|$ then yields
\begin{equation*}
S_P^2\lesssim \sum_{s=0}^{\infty}\frac{t^{d-1
    -\mu_{\Omega}}}{\left(1+ a_*^{1-d}\varepsilon 2^{s}\right)^{N}}\!\!\!\!\sum_{\substack{k\in A^{-1}\mathfrak N^2_P\\2^s\le |Ak|<2^{s+1}}}\!\!\!\!\!|Ak|^{-1-\mu_{\Omega}}\left(\sum_{l=1}^{d-1}
    \left(\frac{|\mathrm O_lAk|}{|Ak|}\right)^{\mathfrak a_1'}\right)^{\!\!\!\!\frac{1}{\mathfrak a_1}-\frac{1}{2}}.
\end{equation*}

We  claim that
\begin{equation}\label{claimn2p}
\sum_{\substack{k\in A^{-1}\mathfrak N^2_P\\ |Ak|\asymp \lambda}}|Ak|^{-1
    -\mu_{\Omega}}\left(\sum_{l=1}^{d-1}
   \left(\frac{|\mathrm O_lAk|}{|Ak|}\right)^{\mathfrak a_1'}\right)^{\frac{1}{\mathfrak a_1}-\frac{1}{2}}
    \lesssim \! a_*^d\lambda^{d-1 -\mu_{\Omega}}.
\end{equation}
We will prove \eqref{claimn2p} later. Using  \eqref{claimn2p} we get
\begin{align}
S_P^2 &\lesssim a_*^d t^{d-1
    -\mu_{\Omega}}\sum_{s=0}^{\infty}\frac{(2^s)^{d-1 -\mu_{\Omega}}}{\left(1+ a_*^{1-d}\varepsilon 2^{s}\right)^{N}} \nonumber \\
  &\lesssim a_*^{(d-1)(d-1-\mu_{\Omega})+d}t^{d-1-\mu_{\Omega}}\varepsilon^{1+\mu_{\Omega}-d}.\label{555}
\end{align}

Using bounds \eqref{444}, \eqref{666} and \eqref{555}, we obtain that
\begin{equation}
R_{A,\varepsilon}(t)
\lesssim a_*^{(d-1)(d-1-\mu_{\Omega})+d}\left(t^{d-1-\nu_{\Omega}}+t^{d-1-\mu_{\Omega}}\varepsilon^{1+\mu_{\Omega}-d}\right).\label{999}
\end{equation}

If $t/a_*$ is sufficiently large, we take
\begin{equation*}
\varepsilon=(t/a_*)^{-\frac{\mu_{\Omega}}{d-\mu_{\Omega}}}.
\end{equation*}
Then $t/2\leq t\pm ca_*\varepsilon\leq 3t/2$. Combining \eqref{777}, \eqref{888} and \eqref{999} yields
\begin{align*}
\left|\#\left(\mathbb Z^d\cap tA\Omega\right)-|\Omega|t^d\right|&\lesssim a_*t^{d-1}\varepsilon+|R_{A,\varepsilon}(t\pm ca_*\varepsilon)|\\
                     &\lesssim a_*^{d^2-d+1}\left(t^{d-1-\nu_{\Omega}}+t^{d-1-\frac{\mu_{\Omega}}{d-\mu_{\Omega}}}\right),
\end{align*}
which is \eqref{eLPP}.

If $t/a_*\asymp 1$ then $ta_i\geq t/a_* \gtrsim 1$. Note that  $tA\Omega$ is contained in an enlarged rectangular box with side lengths $O(ta_1)$, \ldots, $O(ta_d)$. Since such a box contains at most $O(t^d)$ lattice points by trivial estimate, we get
\begin{equation*}
\#\left(\mathbb Z^d\cap tA\Omega\right)\lesssim t^d,
\end{equation*}
which leads to \eqref{eLPP} trivially. This finishes the proof.
\end{proof}

\begin{proof}[Proof of \eqref{claimn2p}]
If $\xi \in \mathfrak N^2_{P}$ and $\dist_{\infty}(\xi, y)\leq 1/(2a_*)$ then  $|\xi|\asymp |y|$ and
\begin{equation*}
\|(\mathrm O_1 \xi,\ldots, \mathrm O_{d-1} \xi)\|_{\ell^{\infty}}=\dist_{\infty}(\xi,\mathbb Rn(P))>3/4,
\end{equation*}
where we have used the definition \eqref{def-mathrmO} of the matrix $\mathrm O$ to obtain the above equality. As a consequence we have
\begin{equation*}
\sum_{l=1}^{d-1}\left(\frac{|\mathrm O_l\xi|}{|\xi|}\right)^{\mathfrak a_1'}\asymp\sum_{l=1}^{d-1}\left(\frac{|\mathrm O_l y|}{|y|}\right)^{\mathfrak a_1'}.
\end{equation*}

Denote by $C_{\xi}$ the open cube in $\mathbb{R}^d$ with center $\xi$, side length $1/a_*$ and all sides parallel to coordinate axes. It is clear that $\{C_{Ak}: k\in\mathbb{Z}^d\}$ are disjoint cubes. Comparing the sum on the left side of \eqref{claimn2p} with an integral, followed by a proper rotation, yields that
\begin{align*}
  &\sum_{\substack{k\in A^{-1}\mathfrak N^2_{P}\\ |Ak|\asymp \lambda}}|Ak|^{-1-\mu_{\Omega}}\left(\sum_{l=1}^{d-1}
   \left(\frac{|\mathrm O_lAk|}{|Ak|}\right)^{\mathfrak a_1'}\right)^{\frac{1}{\mathfrak a_1}-\frac{1}{2}}\\
   &\lesssim a_*^d\int_{|y_d|\asymp \lambda}\int_{|y'|\lesssim\lambda}|y|^{-1-\mu_{\Omega}}
   \left(\sum_{l=1}^{d-1}
   \left(\frac{|y_l|}{|y|}\right)^{\mathfrak a_1'}\right)^{\frac{1}{\mathfrak a_1}-\frac{1}{2}}\,\mathrm{d}y'\mathrm{d}y_d\\
   &\lesssim a_*^d\lambda^{d-1-\mu_{\Omega}},
\end{align*}
as desired.
\end{proof}

Recall that in Section \ref{introduction} we denote by $\Omega_j\subset\mathbb{R}^d$ the intersection of $\Omega$ with the coordinate hyperplane $x_j=0$ and by $|\Omega_j|$ the $(d-1)$-dimensional measure of $\Omega_j$. We sometimes naturally treat $\Omega_j$ as a subset of $\mathbb R^{d-1}$. The following result on the number of lattice points in
$tA \cup_{j=1}^d\Omega_j$ is a consequence of the previous proposition.

\begin{proposition} \label{Corunionforhyper}
 Let $\Omega\subset [-C,C]^d\subset\mathbb R^d$ ($d\geq 2$, $C>0$) be a convex compact domain, which contains the origin as an inner point,  with smooth boundary of finite type. If $t/a_*\geq 1/C$, then
\begin{equation} \label{unionforhyper}
\begin{split}
&\#\left(\mathbb{Z}^d\cap tA \bigcup_{j=1}^d\Omega_j \right)\\
&=\sum_{j=1}^da_j^{-1}|\Omega_j|t^{d-1}+ O\left(  a_*^{d^2-d+1}\left( t^{d-1-\nu_{\Omega}}+t^{d-1-\frac{\mu_{\Omega}}{d-\mu_{\Omega}}}\right)\right),
\end{split}
\end{equation}
where the implicit constant depends only on the domain $\Omega$.
\end{proposition}

\begin{proof}
For $1\leq j\neq l\leq d$, let $A_j$ be the $(d-1)\times(d-1)$ matrix obtained from $A$ by deleting the $j$-th row and column, and $A_{j,l}$ be the $(d-2)\times(d-2)$ matrix  by deleting the $j$-th and $l$-th rows and columns. Denote by  $\Omega_{j,l}$ the intersection of $\Omega$ with the hyperplane $x_j=0$ and $x_l=0$. We sometimes treat $\Omega_{j,l}$ as a subset of  $\mathbb R^{d-2}$. Hence $A_j\Omega_j$ and $A_{j,l}\Omega_{j,l}$ make sense.

It is geometrically evident that
\begin{align*}
&\sum_{j=1}^d\#\left(\mathbb{Z}^{d-1}\cap tA_j\Omega_j\right)-\sum_{1\leq j<l \le d}\#\left(\mathbb{Z}^{d-2}\cap tA_{j,l}\Omega_{j,l}\right)\\
&\le    \#\left(\mathbb{Z}^d\cap tA\bigcup_{j=1}^d\Omega_j\right)
\le\sum_{j=1}^d\#\left(\mathbb{Z}^{d-1}\cap tA_j\Omega_j\right).
\end{align*}
Hence it suffices to find the asymptotics of $\# (\mathbb{Z}^{d-1}\cap tA_j\Omega_j )$ (by Proposition  \ref{LPP})
and estimate the size of $\# (\mathbb{Z}^{d-2}\cap tA_{j,l}\Omega_{j,l} )$. Combining with the above inequality, we will then get the desired asymptotics \eqref{unionforhyper}.

If $d\ge3$,  applying Proposition \ref{LPP} to the domain $\Omega_j\subset\mathbb{R}^{d-1}$ yields
\begin{align*}
  &\quad \left|\#\left(\mathbb{Z}^{d-1}\cap tA_j\Omega_j\right)-a_j^{-1}|\Omega_j|t^{d-1}\right|\\
&=\left|\#\left(\mathbb{Z}^{d-1}\cap \left(t a_j^{-\frac{1}{d-1}}\right)\left(a_j^{\frac{1}{d-1}}A_j\right)\Omega_j\right)-a_j^{-1}|\Omega_j|t^{d-1}\right|\\
&\lesssim a_*^{d^2-d+1}\left(t^{d-2-\nu_{\Omega_j }}+t^{d-2-\frac{\mu_{\Omega_j}}{d-1-\mu_{\Omega_j}}}\right),
\end{align*}
where
\begin{equation*}
 \nu_{\Omega_j}=\min_{P\in\partial\Omega_j}\nu_{\Omega_j}(P)\quad
 \text{and}\quad \mu_{\Omega_j}=1/2+\min_{P\in\partial\Omega_j}\nu^{(2)}_{\Omega_j}(P).
\end{equation*}
By Lemma \ref{factmultitype} we have
\begin{equation*}
1+\nu_{\Omega_j}(P)>\nu_{\Omega}(P) \quad \textrm{for any $P\in \partial\Omega_j$},
\end{equation*}
which gives
\begin{equation*}
    d-2-\nu_{\Omega_j}< d-1-\nu_{\Omega}.
\end{equation*}
We also have
\begin{equation*}
 d-2 -\frac{\mu_{\Omega_j}}{d-1-\mu_{\Omega_j}}<d-2< d-1-\frac{\mu_{\Omega}}{d-\mu_{\Omega}}
\end{equation*}
since $\mu_{\Omega}\le (d-1)/2$. We thus readily get
\begin{equation}
\begin{split}\label{dimd1}
&\quad \left|\#\left(\mathbb{Z}^{d-1}\cap tA_j\Omega_j\right)-a_j^{-1}|\Omega_j|t^{d-1}\right|\\
&\lesssim a_*^{d^2-d+1} \left(t^{d-1-\nu_{\Omega}}+t^{d-1-\frac{\mu_{\Omega}}{d-\mu_{\Omega}}}\right).
\end{split}
\end{equation}
Note that \eqref{dimd1} holds trivially if $d=2$. This provides the asymptotics of $\# (\mathbb{Z}^{d-1}\cap tA_j\Omega_j )$ we need.

As to the size of $\# (\mathbb{Z}^{d-2}\cap tA_{j,l}\Omega_{j,l} )$, we observe that $tA_{j,l}\Omega_{j,l}$ (as a subset of $\mathbb{R}^{d-2}$) is contained in a rectangular box with side lengths $O(ta_1)$,\ldots, $O(ta_{j-1})$,$O(ta_{j+1})$,\ldots, $O(ta_{l-1})$, $O(ta_{l+1})$, \ldots, $O(ta_{d})$. We also know that $ta_i\ge t/a_*\ge 1/C$ and $\det A=1$. By trivial estimate we have
\begin{equation}\label{dimd2}
\begin{split}
\#\left(\mathbb Z^{d-2}\cap tA_{j,l}\Omega_{j,l}\right)&\lesssim (a_ja_l)^{-1} t^{d-2}\\
 &\lesssim a_*^{d^2-d+1} \left(t^{d-1-\nu_{\Omega}}+t^{d-1-\frac{\mu_{\Omega}}{d-\mu_{\Omega}}}\right).
\end{split}
\end{equation}
This provides the size of $\# (\mathbb{Z}^{d-2}\cap tA_{j,l}\Omega_{j,l} )$, thus finishes the proof.
\end{proof}


\begin{theorem}\label{latticeInpositive}
Let $\Omega\subset [-C,C]^d\subset\mathbb R^d$ ($d\geq 2$, $C>0$) be a convex compact domain with smooth boundary of finite type that is symmetric with respect to each coordinate hyperplane (axis if $d=2$). If $t/a_*\ge1/C$, then
\begin{equation}\label{numpositive}
\begin{split}
\#\left(\mathbb{N}^d\cap tA\Omega\right)&=2^{-d}|\Omega|t^d-2^{-d}\sum_{j=1}^da_j^{-1}|\Omega_j|t^{d-1}\\
      &\quad+O\left( a_*^{d^2-d+1}\left(t^{d-1-\nu_{\Omega }}+t^{d-1-\frac{\mu_{\Omega}}{d-\mu_{\Omega}}}\right)\right)
\end{split}
\end{equation}
and
\begin{equation}\label{numnonnega}
\begin{split}
\#\left(\mathbb{Z}^d_+\cap tA\Omega\right)&=2^{-d}|\Omega|t^d+2^{-d}\sum_{j=1}^da_j^{-1}|\Omega_j|t^{d-1}\\
&\quad+O\left( a_*^{d^2-d+1}\left(t^{d-1-\nu_{\Omega }}+t^{d-1-\frac{\mu_{\Omega}}{d-\mu_{\Omega}}}\right)\right),
\end{split}
\end{equation}
where implicit constants depend only on the domain $\Omega$.
\end{theorem}

\begin{proof}
Since the domain $\Omega$ is symmetric we have
\begin{equation*}
\#\left(\mathbb{N}^d\cap tA\Omega\right)=2^{-d}\left(\#\left(\mathbb{Z}^d\cap tA\Omega\right)-\#\left(\mathbb{Z}^d\cap tA\bigcup_{j=1}^d\Omega_j\right)\right).
\end{equation*}
Then \eqref{numpositive} can be obtained from \eqref{eLPP} and  \eqref{unionforhyper}.

It remains to prove \eqref{numnonnega}. For $1\leq j\leq d$ let $P_{j}(1,2,\ldots,d)$ be the collection of all subsets of $\{1,2,\ldots,d\}$ having exactly $j$ elements. For any $S\in P_j(1,2,\ldots,d)$ let
\begin{equation*}
k(S)=\left\{(k_1,\ldots,k_d)\in\mathbb{Z}^d_+: k_i=0\ \textrm{if}\ i\in S; k_i\in \mathbb{N}\textrm{ otherwise}\right\}.
\end{equation*}
Then
\begin{equation*}
\#\left(\mathbb{Z}_+^d\cap tA\Omega\right)=\#\left(\mathbb{N}^d\cap tA\Omega\right)+\sum_{j=1}^{d}\sum_{S\in P_j(1,\ldots,d)}\#\left(k(S)\cap tA\Omega\right).
\end{equation*}
 Notice that
 \begin{equation*}
\sum_{S\in P_1(1,\ldots,d)}\#\left(k(S)\cap tA\Omega\right)=\sum_{j=1}^{d}\#\left({\mathbb{N}^{d-1}\cap tA_{j}\Omega_{j}}\right).
\end{equation*}
By the symmetry of $\Omega_j$, \eqref{dimd1} and \eqref{dimd2}, we have
\begin{align*}
&\quad \sum_{j=1}^{d}\#\left({\mathbb{N}^{d-1}\cap tA_{j}\Omega_{j}}\right)\\
&=\sum_{j=1}^{d}2^{-(d-1)}a_j^{-1}|\Omega_j|t^{d-1}+O\left( a_*^{d^2-d+1} \left(t^{d-1-\nu_{\Omega}}+t^{d-1-\frac{\mu_{\Omega}}{d-\mu_{\Omega}}}\right)\right).
\end{align*}
By \eqref{dimd2} we also have
\begin{align*}
\sum_{j=2}^d\sum_{S\in P_j(1,\ldots,d)}\#\left(k(S)\cap tA\Omega\right)&\lesssim\sum_{1\le l<m\le d}\#\left({\mathbb{Z}^{d-2}\cap tA_{l,m} \Omega_{l,m}}\right)\\
&\lesssim a_*^{d^2-d+1} \left(t^{d-1-\nu_{\Omega}}+t^{d-1-\frac{\mu_{\Omega}}{d-\mu_{\Omega}}}\right).
\end{align*}
Then \eqref{numnonnega} follows from the above four equalities and \eqref{numpositive}.
\end{proof}


\section{Proof of Theorem \ref{main1}} \label{sec3}

With results of lattice point counting established, we follow a standard procedure to prove Theorem \ref{main1}. We refer readers to \cite{LL16} and also \cite[Section 4]{GW18} for this procedure.

We first consider the case $A(t)\in \mathfrak A_{\Omega}(t)$. We set a diagonal matrix
\begin{equation}  \label{def-B}
B=\diag\left(\frac{|\Omega_1|}{\sqrt[d]{|\Omega_1|\cdots|\Omega_d|}},
\ldots,\frac{|\Omega_d|}{\sqrt[d]{|\Omega_1|\cdots|\Omega_d|}}\right).
\end{equation}
Applying \eqref{numpositive} with this stretching factor $B$ yields
\begin{equation}\label{tBD}
\begin{split}
\#\left(\mathbb{N}^d\cap tB\Omega\right)&=2^{-d}|\Omega|t^d-2^{-d}d\sqrt[d]{|\Omega_1|\cdots|\Omega_d|}t^{d-1}\\
              &\quad +O\left(t^{d-1-\nu_{\Omega}}+t^{d-1-\frac{\mu_{\Omega}}{d-\mu_{\Omega}}}\right),
\end{split}
\end{equation}
which leads to
\begin{equation} \label{idlowwer}
       \#\left(\mathbb{N}^d\cap tB\Omega\right)\geq 2^{-d}|\Omega|t^d-2^{1-d}d\sqrt[d]{|\Omega_1|\cdots|\Omega_d|}t^{d-1}
\end{equation}
for sufficiently large $t$.

Since $A(t)\in \mathfrak A_{\Omega}(t)$, we have
\begin{equation*}
t/a_*(t)\ge 1/C,
\end{equation*}
where $a_*(t)=\|A(t)^{-1}\|_{\infty}$ and $C>0$ is a constant satisfying $\Omega\subset[-C,C]^d$, otherwise $tA(t)\Omega$ does not contain any positive lattice point. Then Lemma \ref{general} gives
\begin{equation}\label{optimalupper}
\#\left(\mathbb{N}^d\cap tA(t)\Omega\right)\leq 2^{-d}|\Omega|t^d-ca_*(t)t^{d-1},
\end{equation}
where $c$ is a positive constant depending only on the domain $\Omega$.

Combining  \eqref{idlowwer}, \eqref{optimalupper} and
\begin{equation}\label{lll}
\#\left(\mathbb{N}^d\cap tB\Omega\right)\leq\#\left(\mathbb{N}^d\cap tA(t)\Omega\right)
\end{equation}
yields that
\begin{equation*}
a_*(t)\leq 2^{1-d}d\sqrt[d]{|\Omega_1|\cdots|\Omega_d|}/c,
\end{equation*}
namely, $a_*(t)$ is uniformly bounded from above for sufficiently large $t$.

Applying \eqref{numpositive} with the stretching factor $A(t)$ gives
\begin{equation}\label{tA(t)D}
\begin{split}
\#\left(\mathbb{N}^d\cap tA(t)\Omega\right)&=2^{-d}|\Omega|t^d-2^{-d}\sum_{j=1}^d a_j(t)^{-1}|\Omega_j|t^{d-1}\\
&\quad +O\left(t^{d-1-\nu_{\Omega}}+t^{d-1-\frac{\mu_{\Omega}}{d-\mu_{\Omega}}}\right).
\end{split}
\end{equation}
Combining \eqref{tBD}, \eqref{tA(t)D} and \eqref{lll} yields that
\begin{equation*}
\sum_{j=1}^d a_j(t)^{-1}\frac{|\Omega_j|}{\sqrt[d]{|\Omega_1|\cdots|\Omega_d|}}\leq d+O\left(t^{-\nu_{\Omega}}+t^{-\frac{\mu_{\Omega}}{d-\mu_{\Omega}}}\right).
\end{equation*}
Then the desired convergence \eqref{optpositive} follows easily from an elementary result in \cite[Lemma B.1]{GW18}. This completes the proof of the first case.


The second case $\widetilde{A}(t)\in \widetilde{\mathfrak A}_{\Omega}(t)$ can be proved similarly. We sketch its proof. Applying \eqref{numnonnega} with the matrix $B$ (defined by \eqref{def-B}) yields
\begin{align}\label{nonnegative-idlowwer}
 \#\left(\mathbb{Z}_+^d\cap tB\Omega\right)\leq 2^{-d}|\Omega|t^d+2^{1-d}d\sqrt[d]{|\Omega_1|\cdots|\Omega_d|}t^{d-1}
\end{align}
for sufficiently large $t$. We also have
\begin{equation}\label{lllnonnegat}
\#\left(\mathbb{Z}^d_+\cap tB\Omega\right)\geq\#\left(\mathbb{Z}^d_+\cap t\widetilde{A}(t)\Omega\right).
\end{equation}

Let $\tilde{a}_*(t)=\|\widetilde{A}(t)^{-1}\|_\infty$. We claim that if $t$ is sufficiently large, then
\begin{equation*}
t/\tilde{a}_*(t)\geq 1/C
\end{equation*}
with the same constant $C$ aforementioned. Indeed, if $t/\tilde{a}_*(t)<1/C$ then $\mathbb{Z}^d_+\cap t\widetilde{A}(t)\Omega$ is contained in $t\widetilde{A}_j(t)\Omega_j$ for some $1\leq j\leq d$, where $\widetilde{A}_j(t)$ is the $(d-1)\times (d-1)$ matrix obtained from $\tilde A(t)$ by removing its $j$-th row and column. Hence
\begin{align*}
\#\left(\mathbb{Z}^d_+\cap t\widetilde{A}(t)\Omega\right)&\geq 2^{-(d-1)}|t\widetilde{A}_j(t)\Omega_j|\\
&>2^{1-d}C|\Omega_j| t^d,
\end{align*}
where in the second inequality we have used $t/\tilde{a}_*(t)<1/C$.
Since $\Omega \subsetneq \Omega_j\times [-C,C]$, we have $2C|\Omega_j|>|\Omega|$. If $t$ is sufficiently large we then have
\begin{equation*}
\#\left(\mathbb{Z}^d_+\cap t\widetilde{A}(t)\Omega\right)\!\!>\!\!2^{-d}|\Omega| t^d+2^{1-d}d\sqrt[d]{|\Omega_1|\cdots|\Omega_d|}t^{d-1}\!\!\geq\!\! \#\left(\mathbb{Z}^d_+\cap tB\Omega\right)
\end{equation*}
by \eqref{nonnegative-idlowwer}. This contradicts with \eqref{lllnonnegat}.

It is then easy to show that $\tilde{a}_*(t)$ is uniformly bounded, as a consequence of \eqref{nonnegative-idlowwer}, \eqref{lllnonnegat} and Lemma \ref{general}. We next apply \eqref{numnonnega} to $\#(\mathbb{Z}_+^d\cap t\tilde A(t)\Omega)$ and $\#(\mathbb{Z}_+^d\cap tB\Omega)$ and use \eqref{lllnonnegat} and \cite[Lemma B.1]{GW18} to finish the proof.  \qed

\appendix

\section{Multitype}\label{appmultitype}

In the appendix we compare the multitypes of $\partial\Omega_j$  and $\partial \Omega$ at a common point $P\in \partial\Omega_j\subset\partial \Omega$. The result is a direct consequence of the definition of multitype (see for example  \cite[P. 155--156]{ISS02}), which  says that the $i$-th component of the multitype of $\partial\Omega_j$ at $P$ is not greater than the $(i+1)$-th component of the multitype of $\partial \Omega$ at $P$.

\begin{lemma}\label{factmultitype}
Let $\Omega\subset \mathbb R^{d}$ $(d\geq 3)$ be a convex compact domain with smooth boundary of finite type and $\Omega_j\subset \mathbb{R}^{d-1}$, $1 \le j \le d$, the intersection of $\Omega$ with the coordinate hyperplane  $x_j=0$. If $(\tilde{\mathfrak a}_1,\ldots,\tilde{\mathfrak a}_{d-2})$ and  $(\mathfrak a_1,\ldots,\mathfrak a_{d-1})$ are multitypes of $\partial\Omega_j$ and $\partial\Omega$ at $P\in \partial\Omega_j $ respectively, then for any $1\le i\le d-2$ we have
\begin{equation*}
\tilde{\mathfrak a}_i\le \mathfrak a_{i+1}.
\end{equation*}
\end{lemma}

\begin{proof}
We first briefly recall the definition of multitype. Let
\begin{equation*}
\{u_1,\ldots, u_{d-1}, -n(P)\}
\end{equation*}
be an orthonormal basis with the same orientation as $\{e_1,\ldots,e_d\}$ and $u_1,\ldots,u_{d-2}\in T_P(\partial\Omega_j)$ and $u_{d-1}\in T_P(\partial\Omega)$. Then the boundary $\partial \Omega$  in a small neighborhood of $P$ can be parameterized by
\begin{equation}\label{localparametrize1}
 \Gamma(V)=P+V-\Phi(V) n(P),
\end{equation}
where $V=\sum_{i=1}^{d-1}x_iu_i\in T_P(\partial\Omega)$ and
\begin{equation*}
\Phi(V)=\Phi(x_1,\ldots,x_{d-1}).
\end{equation*}
It is obviously that $\Phi(0)=\nabla\Phi(0)=0$. For any $m\ge 2$, define
\begin{equation*}
S_P^m=\left\{x\in\mathbb R^{d-1}:\sum_{j=2}^{m}\left|D_x^j\Phi(0)\right|=0\right\},
\end{equation*}
where
\begin{equation*}
D_x^j\Phi(0)=\left.\left(\frac{\partial}{\partial t}\right)^j\Phi(tx)\right|_{t=0}
\end{equation*}
is the $j$-th derivative of $\Phi$ at the origin in the direction $V=\sum_{i=1}^{d-1}x_iu_i$. Then there are at most $d-1$ even numbers
\begin{equation*}
2\le m_1<m_2<\cdots<m_k,\quad 1\le k\le d-1
\end{equation*}
such that the sequence
\begin{equation}\label{spacesequence}
\{0\}=S_P^{m_k}\subsetneq\cdots\subsetneq S_P^{m_1}\subsetneq S_P^{m_0}=\mathbb R^{d-1}
\end{equation}
is maximal in the sense that $S_P^n=S_P^{m_{j-1}}$ if $m_{j-1}\le n<m_j$ (see \cite[P. 1270]{schulz}). Here  $m_0=m_1-1$.
 For $1\le j\le k$ we define
\begin{equation*}
\mathfrak a_i=m_j\quad\text{if}\quad d-1-\dim S_P^{m_{j-1}}<i\le d-1-\dim S_P^{m_j}.
\end{equation*}
Then $\mathfrak a=(\mathfrak a_1,\ldots,\mathfrak a_{d-1})$ is the multitype of $\partial \Omega$ at $P$.
Notice that the multitype $\mathfrak a$ is independent of the choice of the orthonormal basis $\{u_1,\ldots,u_{d-1}\}$. Furthermore, the convexity of $\Phi$ makes $S_P^{m_j}$'s the linear subspaces of $\mathbb R^{d-1}$.  Let $W_j$ be the orthogonal complement of $S_P^{m_j}$ in $S_P^{m_{j-1}}$. Then
\begin{equation}\label{directsum1}
S_P^{m_0}=S_P^{m_j}\oplus W_j\oplus W_{j-1}\oplus\cdots\oplus W_1.
\end{equation}
We observe that the dimension of $W_j$ is the number of $m_j$ appearing in the multitype $\mathfrak a$.

Correspondingly, by our choice of $\{u_1,\ldots,u_{d-2}\}$, the boundary $\partial\Omega_j$ in a small neighborhood of $P$ can be parameterized by
\begin{equation*}\label{localparametrize2}
\Gamma(V)=P+V-\tilde\Phi(V)n(P),
\end{equation*}
where  $V=\sum_{i=1}^{d-2}x_iu_i\in T_P(\partial\Omega_j)$ and
\begin{equation*}
\tilde\Phi(V)=\tilde\Phi(x_1,\ldots,x_{d-2})=\Phi(x_1,\ldots,x_{d-2},0).
\end{equation*}
Notice that for any $\tilde x=(x_1, \ldots, x_{d-2})\in \mathbb R^{d-2}$ and $j\ge 2$,
\begin{align*}
D_{\tilde x}^j\tilde\Phi(0)
&=\left.\left(\frac{\partial}{\partial t}\right)^j\Phi(tx_1,\ldots,tx_{d-2},0)\right|_{t=0}
=D_{x}^j\Phi(0),
\end{align*}
where $x=(x_1, \ldots, x_{d-2},0)$, namely  the $j$-th derivative of $\tilde \Phi$ at the origin in the direction $V=\sum_{i=1}^{d-2}x_iu_i$ equals  the $j$-th derivative of $\Phi$ at the origin in the direction $V=\sum_{i=1}^{d-2}x_iu_i+0u_{d-1}$.
Hence  all $\tilde{\mathfrak a}_j$'s are chosen from $\{m_1,\ldots,m_k\}$ by the maximization of the space sequence \eqref{spacesequence} and the definition of multitype. For every $m_j$, let
\begin{equation*}
\tilde S_P^{m_j}=\left\{\tilde x\in \mathbb R^{d-2}:\sum_{j=2}^{m_j}|D_{\tilde x}^j\tilde\Phi(0)|=0\right\}
\end{equation*}
and $\widetilde W_j$  be the orthogonal complement of $\tilde S_P^{ m_j}$ in $\tilde S_P^{ m_{j-1}}$.
Then we have
\begin{equation}\label{directsum2}
\tilde S_P^{m_0}=\tilde S_P^{m_j}\oplus \widetilde W_j\oplus \widetilde W_{j-1}\oplus\cdots\oplus \widetilde W_1.
\end{equation}

 Notice that
\begin{equation}\label{multitypefacts1}
\dim S_P^{ m_0}=d-1,\quad \dim \tilde S_P^{m_0}=d-2
\end{equation}
and for any $1\le j\le k$,
\begin{equation}\label{multitypefacts2}
\dim S_P^{m_j}\ge \dim \tilde S_P^{ m_j}.
\end{equation}
 Then  combining  \eqref{directsum1}--\eqref{multitypefacts2}  yields that for any $1\le j\le k$,
\begin{equation*}
\dim(W_1\oplus\cdots\oplus W_j)-1\le \dim(\widetilde W_1\oplus\cdots\oplus\widetilde W_j),
\end{equation*}
namely   the number of $m_1$ appearing in $\tilde {\mathfrak a}$  is no less than the number of $m_1$ appearing in $\mathfrak a$ minus $1$, and  the same is true for the number of $m_1,m_2$ and more generally for the number of $m_1,m_2,\ldots,m_j$ with $1\le j\le k$.  Hence we obtain the desired result. Notice that in some special cases we may have $\dim \widetilde W_j=0$. Then $m_j$ will not appear in $\tilde{\mathfrak a}$. But this does not affect our conclusion.
\end{proof}

%

%


\end{document}